\documentclass[11pt]{amsart}
\usepackage{latexsym,amscd,amssymb,epsfig} 
\usepackage[dvips]{color}
\usepackage[all]{xy}

\theoremstyle{plain}
  \newtheorem{theorem}{Theorem}[section]
  \newtheorem{proposition}[theorem]{Proposition}
  \newtheorem{lemma}[theorem]{Lemma}
  \newtheorem{corollary}[theorem]{Corollary}
  
\theoremstyle{definition}

\theoremstyle{remark}
  \newtheorem{remark}[theorem]{Remark}

\numberwithin{equation}{section}

\def\clap#1{\hbox to 5pt{\hss$#1$\hss}}

\def\umapright#1{\smash{
   \mathop{\longrightarrow}\limits^{#1}}}

\def\rmapdown#1{\Big\downarrow\rlap
   {$\vcenter{\hbox{$\scriptstyle#1$}}$}}

\def\tempbaselines
{\baselineskip22pt\lineskip3pt
   \lineskiplimit3pt}
\def\diagram#1{\null\,\vcenter{\tempbaselines
\mathsurround=0pt
    \ialign{\hfil$##$\hfil&&\quad\hfil$##$\hfil\crcr
      \mathstrut\crcr\noalign{\kern-\baselineskip}
  #1\crcr\mathstrut\crcr\noalign{\kern-\baselineskip}}}\,}

\def\pullback#1&#2&#3&#4&#5&#6&#7&#8&{
\diagram{#1&\umapright{#2}&#3\cr
\rmapdown{#4}&&\rmapdown{#5}\cr
#6&\umapright{#7}&#8\cr}}

\def\calC{{\mathcal C}}
\def\calD{{\mathcal D}}
\def\calE{{\mathcal E}}

\def \End{\mathop{\rm End}\nolimits} 
\def \Ext{\mathop{\rm Ext}\nolimits} 
\def \Fun{\mathop{\rm Fun}\nolimits} 

\def \Hom{\mathop{\rm Hom}\nolimits} 
\def \uHom{\mathop{\underline{\rm Hom}}\nolimits}
\def \Id{\mathop{\rm Id}\nolimits}

\def \op{{\mathop{\rm op}\nolimits}}

\def\Rad{\mathop{\rm Rad}\nolimits}
\def\Serre{{\mathcal S}}

\def \hatH{\mathop{\widehat{\rm H}}\nolimits}

\def\stmod{\mathop{\bf stmod}\nolimits}

\def\clap#1{\hbox to 0pt{\hss$#1$\hss}}
\def\lra{{\longrightarrow}}

\def\ZZ{{\mathbb Z}}

\begin{document}

\title[The Graded Center]{The Graded Center of a Triangulated Category}

\author{Jon F. Carlson}
\email{jfc@math.uga.edu}
\address{Department of Mathematics\\
University of Georgia\\
Athens, GA 30602, USA}

\author{Peter Webb}
\email{webb@math.umn.edu}
\address{School of Mathematics\\
University of Minnesota\\
Minneapolis, MN 55455, USA}



\begin{abstract}
  With applications in mind to the representations and cohomology
  of block algebras, we examine elements of the graded center of
  a triangulated category when the category has a Serre functor.
  These are natural transformations from the identity functor to
  powers of the shift functor that commute with the shift functor
  We show that such natural transformations which have  support in
  a single shift orbit of indecomposable objects are necessarily
  of a kind previously constructed by Linckelmann. Under further
  conditions, when the support is contained in only finitely many
  shift orbits, sums of transformations of this special kind
  account for all possibilities.

  Allowing infinitely many shift orbits in the support, we construct
  elements of the graded center of the stable module category of
  a tame group algebra of a kind that cannot occur with wild block algebras.
  We use functorial methods extensively in the proof, developing some
  of this theory in the context of triangulated categories.
\end{abstract}

\dedicatory{To the memory of a wonderful friend, Laci Kov\'acs}

\maketitle

\section{Introduction}
The graded center of a triangulated category $\calC$ is the set of
natural transformations $\Id_\calC \to \Sigma^n$ that commute with the
shift functor $\Sigma$ up to a sign $(-1)^n$. 
In \cite{Lin}, Linckelmann investigated
the graded center of a block algebra of a finite group. His main result
showed that the graded center of the derived category of a block is,
modulo a nilpotent ideal, noetherian over the cohomology ring of the
block. Along the way, Linckelmann showed that there is a large ideal
of nilpotent elements in the graded center generated by elements
in degree minus one that are supported on only a single $\Sigma$-orbit
of modules. This result was extended by Linckelmann and Stancu to
obtain elements in all degrees each of which is supported on only a single
module that is periodic of period one. 

Unltimately we would like to be able to characterize the nilpotent 
elements in the graded center. In that direction, a natural question to ask is
whether there can be elements of the graded center that are non-trivially
supported on more than one $\Sigma$-orbit? By this we mean, do there
exist elements in the graded center that vanish on all but a finite 
number of $\Sigma$-orbits, but which have a non-zero composition with some
non-isomorphism?  
A main purpose of this paper is to show that 
the answer to that question is generally negative. 

For the most part, we work in a Hom-finite,
Krull-Schmidt, $k$-linear triangulated  category $\calC$
that is Calabi-Yau. For
$F: \calC \to \calC$ an endofunctor, we define the support of a natural
transformation $\alpha: \Id_\calC \to F$ to be the set of isomorphism
classes of objects $U$ with $\alpha_U: U \to F(U)$ not zero. We prove
that the support of such an $\alpha$ is a single object if and only if
$F$ is the Serre functor and for any object $M$, $\alpha_M$ is an almost
vanishing morphism. Thus such natural transformations have the same
form as those contructed by Linckelmann. We show that, under some
reasonable assumptions on the Auslander-Reiten quiver, the support of
an element of the graded center, which is supported non-trivially
on more than one shift orbit in a
component of the Auslander-Reiten quiver, has the entire component in
its support. Moreover, if $\calC$ is the stable category of a group
algebra of a $p$-group of wild representation type, 
then such a element $\psi$ of the graded center has the property that
there exists a map $\gamma:U \to M$ of indecomposable modules 
such that $\psi_M \gamma \neq 0$ and $\gamma$
is not a composition of a finite number of irreducible morphisms. 
We show, by the example
of a finite group with dihedral Sylow 2-subgroup, that this
requirement does not hold if the group has tame representation type. 

Throughout the paper we assume that $k$ is an algebraically closed 
field and that $\calC$ is a Hom-finite, Krull-Schmidt, $k$-linear 
triangulated  category with shift $\Sigma$. For background on 
Auslander-Reiten theory we refer to standard texts such as \cite{ASS}.

Both authors are grateful for support from the Simons Foundation. The
first author also thanks NSA for support during part of the time
when this paper was written. 

\section{Linear functors on a triangulated category}

We present some preliminaries on the category of linear functors
defined on a triangulated category.  Let $\Fun^\op\calC$ denote
the category of contravariant
$k$-linear functors from $\calC$ to $k$-vector spaces.  The first
four results of this section are well known.
While they are usually stated for
functors on module categories,  they hold
for $k$-linear functors on $k$-linear categories (and even for
additive functors on additive categories). In particular they hold
for $\Fun^\op\calC$, ignoring the triangulated structure of $\calC$.
There are proofs of these results in \cite{ASS} stated in terms of
functors on the module category of a ring, but the arguments there
work for functors on an additive category without change. Our 
purpose is to point out that these results hold in the generality 
we consider here.

Recall our assumption that $\calC$ is a Hom-finite,
Krull-Schmidt, $k$-linear triangulated  category with shift $\Sigma$,

\begin{proposition}
\label{representable-functors}
The category $\Fun^\op\calC$ is an abelian category. 
Moreover, for each indecomposable object $M$ the representable
functor $\Hom_\calC(-,M)$ is indecomposable and projective,
with endomorphism ring isomorphic to $\End_\calC(M)$.
\end{proposition}

\begin{proof}
See IV.6.2(a), IV.6.4(a) and A.2.9 of \cite{ASS}. It is standard 
that $\Fun^\op\calC$ is an abelian category in which  kernels, 
cokernels and exactness are determined by evaluation at  the objects 
of $\calC$. The statements about representable functors  are a 
consequence of the linear form of Yoneda's Lemma, which in our 
usage says that any morphism from $\Hom_\calC(-,M)$ to 
$\Hom_\calC(-,N)$ is induced from a morphism from $M$ to $N$. 
\end{proof} 

The simple functors $s^M\in \Fun^\op\calC$ are defined
in \cite[IV.6.7]{ASS}. For each indecomposable object $M$
of $\calC$ we have
$$
s^M=\Hom_\calC(-,M)/\Rad_\calC(-,M)
$$
where $\Rad_\calC(-,M)$ is the radical of $\Hom_\calC(-,M)$,
the subfunctor whose value at an object $X$ is the set of non-isomorphisms
from $X$ to $M$. These functors
have the description
$$
s^M(N)=
\begin{cases}
k&\hbox{if }M\cong N,\cr
0&\hbox{otherwise.}\cr
\end{cases}
$$
The next result is also well known in the context of functors
on module categories, and the proof given in the reference 
carries through verbatim.

\begin{proposition}
\begin{enumerate}
\item The simple objects in $\Fun^\op\calC$ are all of the form 
$s^M$ for some indecomposable object $M$ in $\calC$. The relation
$M \leftrightarrow s^M$ give a one-to-one correspondence 
between isomorphism types of indecomposable
  objects in $\calC$ and isomorphism
  types of simple objects in $\Fun^\op\calC$. 
\item The quotient map $\Hom_\calC(-,M)\to s^M$ is a projective
  cover, having kernel $\Rad_\calC(-,M)$, which is the unique
  maximal subfunctor of $\Hom_\calC(-,M)$. 
\end{enumerate}
\end{proposition}

\begin{proof}
See \cite[IV.6.8]{ASS}.
\end{proof}

We say that a functor $F$ is finitely generated if there is an
epimorphism $\Hom_\calC(-,M)\to F$ for some object $M$. We say
that $F$ is finitely presented if there is an exact sequence of
functors $\Hom_\calC(-,M_1)\to \Hom_\calC(-,M_0)\to F\to 0$.

The next result is, again, usually only stated for functors
on module categories, but it is also true for functors on
additive or k-linear categories.

\begin{proposition}
\label{fg-projectives}
The representable functors $\Hom_\calC(-,M)$, where $M$ is
indecomposable, are a complete list of the indecomposable
finitely generated projective functors in $\Fun^\op\calC$.
\end{proposition}

\begin{proof}
See \cite[IV.6.5]{ASS}
\end{proof}

We say that $s^M$ is a \textit{composition factor} of a
functor $F$ if there are subfunctors $F_0\subset F_1$ of $F$
so that $F_1/F_0\cong s^M$.

\begin{corollary}
\label{composition-factor-criterion}
Let $F$ be a functor in $\Fun^\op\calC$. Then $F$ has $s^M$ as
a composition factor if and only if $F(M)\ne 0$. Furthermore,
$F$ has finite composition length if and only if   $F$ is non-zero
on only finitely many isomorphism classes of indecomposable
objects of $\calC$, where its value is finite dimensional.
\end{corollary}

\begin{proof}
  If $F$ has $s^M$ as a composition factor then, since $s^M(M)\ne 0$,
  we must have $F(M)\ne 0$. Conversely, if $F(M)\ne 0$ then
  by Yoneda's Lemma there is a non-zero morphism $\Hom_\calC(-,M)\to F$
  showing that the unique simple quotient $s^M$ of $\Hom_\calC(-,M)$
  appears as a composition factor of $F$. The statement about
  finite composition length of $F$ follows from the fact that
  each simple functor is non-zero on a single isomorphism class
  of indecomposable objects, where its value has dimension 1.
\end{proof}

We turn now to a result for triangulated categories which is not
the same as for module categories.  It is well known that finitely
presented functors on a module category have projective dimension
at most 2, because $\Hom$ is left exact on 
such a category \cite[Prop. 4.2]{AR1}. The situation 
for functors on a triangulated category is quite different.

\begin{proposition}
\label{finitistic-dimension}
The only finitely presented functors of finite projective dimension
in $\Fun^\op\calC$ are the representable functors. Equivalently, 
the only monomorphisms between representable functors are split.
\end{proposition}

\begin{proof}
  Given a presentation $\Hom_\calC(-,M_1)\to \Hom_\calC(-,M_0)\to F\to 0$
  of a functor $F$ the morphism between the representable functors
  comes from a morphism $\alpha:M_1\to M_0$ in $\calC$, by Yoneda's
  Lemma. Complete this morphism to a triangle and rotate it, to
  get a triangle
$$
M_2\xrightarrow{\beta}M_1\xrightarrow{\alpha}M_0
\xrightarrow{\gamma}\Sigma M_2.
$$
We get a long exact sequence of representable functors which
has $F$ as one of its (co)kernels, giving a long exact sequence
$$
\begin{matrix}
&&\cdots&\lra &\Hom_\calC(-,\Sigma^{-1}M_1)&\lra 
 &\Hom_\calC(-,\Sigma^{-1}M_0)\cr
&\lra &\Hom_\calC(-,M_2)&\lra &\Hom_\calC(-,M_1)&\lra &\Hom_\calC(-,M_0)\cr
&\lra& F&\lra &0,&&&\cr
\end{matrix}
$$
which is a projective resolution of $F$.  By a standard result
in homological algebra,  $F$ has finite projective dimension if
and only if at some stage the kernel in this resolution is
projective. Note that every kernel is  finitely generated,
because it is the image of the next term in the resolution. Thus,
if $F$ has finite projective dimension, then one of the morphisms
$\Hom_\calC(-,U)\to \Hom_\calC(-,V)$ in the sequence factors
as a surjection onto a projective functor followed by an injection from
the projective functor as follows:
$$
\Hom_\calC(-,U)\to \Hom_\calC(-,X)\to \Hom_\calC(-,V)
$$
corresponding to morphisms
$U\xrightarrow{\phi} X\xrightarrow{\theta} V$ in $\calC$. 
Here we are using the fact that projective functors are representable,
by Proposition~\ref{fg-projectives}.
Because the first of these maps of functors is surjective,
the identity $1_X$ is an image of a map $X\to U$ after
composition with $\phi$, so that $\phi$ is a split epimorpism. The
injectivity of the second map of functors is exactly the
definition that $\theta$ is a monomorphism. In a triangulated
category all monomorphisms are split, so that $\theta$ is
a split monomorphism. It follows from this that there are decompositions
$U\cong U_1\oplus X$ and $V\cong X\oplus V_1$ so that the
morphism $U\to V$ is the identity on $X$ and 0 on $U_1$.
Transferring back to the original triangle we see that it
is the sum of two triangles, one of the form
$\Sigma^nX\xrightarrow{1}\Sigma^nX\to 0 \to \Sigma^{n+1}X$
for some $n$,  and the other with zero as one of its three
morphisms. The first of these produces contractible summands
of the resolution of $F$, and the second produces a
resolution which is split everywhere, showing that $F$ is
projective because the final map $\Hom_\calC(-,M_0)\to F$
must split.

The equivalence with the statement that monomorphisms between 
representable functors are split is immediate. If there is a 
non-split such monomorphism, then its cokernel has projective 
dimension 1 and is not projective. On the other hand, any 
non-projective functor of finite projective dimension gives 
rise to a functor of projective dimension 1 (appearing at the 
end of the finite projective resolution), and this is presented 
by a non-split monomorphism of projectives.
\end{proof}

We characterize the existence of Auslander-Reiten triangles in 
terms of finite presentability of the corresponding simple functors. 
The result is familiar for functors on module categories, but less 
so for functors on triangulated categories.

\begin{proposition}
\label{simple-functors-finitely-presented}
Let $\calC$ be a Hom-finite, Krull-Schmidt triangulated category, and 
let $M$ be an indecomposable object in $\calC$. The simple functor 
$s^M$ is finitely presented if and only if there is an Auslander-Reiten 
triangle $U\to V\to M\to \Sigma U$. When there exists such an 
Auslander-Reiten triangle the map of representable functors 
$\Hom_\calC(-,M)\to \Hom_\calC(-,\Sigma U)$ has $s^M$ as its image.
\end{proposition}

\begin{proof}
If there is such an Auslander-Reiten triangle the long exact sequence
$$
\cdots\to \Hom_\calC(-,U)\to \Hom_\calC(-,V)\to 
\Hom_\calC(-,M)\to \Hom_\calC(-,\Sigma U)\to\cdots
$$
provides the start of a resolution
$$
\Hom_\calC(-,V)\to \Hom_\calC(-,M)\to s^M\to 0
$$
because the lifting property of the Auslander-Reiten triangle coupled 
with the fact that it is not split translates to the statement that 
the cokernel of 
$\Hom_\calC(-,V)\to \Hom_\calC(-,M)$ is $s^M$, 
which is also the image of 
$\Hom_\calC(-,M)$ in $\Hom_\calC(-,\Sigma U)$. 
This shows that $s^M$ is finitely presented. 

Conversely, if $s^M$ is finitely presented by a 3-term exact sequence 
of this form, then the morphism $\Hom_\calC(-,V)\to \Hom_\calC(-,M)$ comes 
from a morphism $V\to M$ in $\calC$ which we may extend to a triangle 
$U\to V\to M\to \Sigma U$. This triangle satisfies the Auslander-Reiten 
lifting property at $M$, and the morphism $M\to \Sigma U$ is not zero 
since $V\to M$ is not a split epimorphism. The triangle gives rise to 
a long exact sequence of representable functors
of the kind at the start of this proof. If the kernel of 
$\Hom_\calC(-,V)\to \Hom_\calC(-,M)$ has a non-zero projective 
direct summand, then such a summand is finitely presented and 
hence representable. By Proposition~\ref{finitistic-dimension}, 
it splits off from $\Hom_\calC(-,V)$, as well as from 
$\Hom_\calC(-,U)$. Thus we can remove such a summand and may 
assume that $\Hom_\calC(-,V)$ is a projective cover of 
$\Rad_\calC(-,M)$. In this case the morphism $V\to M$ is minimal 
right almost split, in the terminology of \cite{ASS}. It is
proven by Happel~\cite[page 36]{Hap} (in the dual case of 
a minimal left almost split morphism) that the third term U 
in the triangle is indecomposable, and hence the triangle is 
an Auslander-Reiten triangle.
\end{proof}

Given a Hom-finite, Krull-Schmidt triangulated category
$\calC$ over $k$, a \textit{Serre functor} on $\calC$ is a
self-equivalence $\Serre:\calC\to\calC$ for which there
are bifunctorial isomorphisms
$$
D\Hom_\calC(X,Y)\cong \Hom_\calC(Y,\Serre(X))\quad
\hbox{for all}\quad X,Y\in\calC.
$$
Here $D(U)=\Hom_k(U,k)$ is the vector space duality.
It was shown in \cite{RvdB} that $\calC$ has a Serre
functor $\Serre$ if and only if $\calC$ has Auslander-Reiten triangles,
and that the Auslander-Reiten triangles have the form
$$
\Sigma^{-1}\Serre(U)\xrightarrow{\alpha} V\xrightarrow{\beta}
U\xrightarrow{\gamma}\Serre(U)
$$
with Auslander-Reiten translate $\tau=\Sigma^{-1}\Serre$.

We now point out that the presence of a Serre functor on
$\calC$ makes $\Fun^\op\calC$ into a self-injective category.
We will use, particularly, the fact that representable
functors for indecomposable objects have simple socles.

\begin{proposition} \label{prop:injective}
  Let $\calC$ be a Hom-finite, Krull-Schmidt triangulated
  category with Serre functor $\Serre$. Then each
  representable functor $\Hom_\calC(-,M)$ is injective
  (as well as projective), with simple socle $s^{\Serre^{-1}(M)}$.
\end{proposition}

\begin{proof}
For each object $X$ we have
$$
\Hom_\calC(X,M)\cong D\Hom_\calC(M,\Serre(X))\cong
D\Hom_\calC(\Serre^{-1}(M),X).
$$
Because $\Hom_\calC(\Serre^{-1}(M),-)$ is a projective covariant
functor on $\calC$ it follows that
$$
D\Hom_\calC(\Serre^{-1}(M),-)\cong \Hom_\calC(-,M)
$$
is an injective contravariant functor on $\calC$, as
well as being projective. Now if
$$
\Sigma^{-1}M\to E\to \Serre^{-1}(M)\to M
$$
is an Auslander-Reiten triangle, the image of
\[
\Hom_\calC(-, \Serre^{-1}(M))\to \Hom_\calC(-,M)
\]
 is the
simple functor $s^{\Serre^{-1}(M)}$ by 
Proposition~\ref{simple-functors-finitely-presented}.
This is the socle.
\end{proof}

We identify composition factors of functors
in $\Fun^\op\calC$ in the spirit of \cite{AR2}.

\begin{corollary}
\label{composition-factors}
Assume that $\calC$ has a Serre functor $\Serre$. Let
$U$ and $M$ be indecomposable objects of $\calC$. The
following are equivalent:
\begin{enumerate}
\item The functor $s^U$ is a composition factor of $\Hom_\calC(-,M)$.
\item there is a non-zero morphism $U\to M$.
\item there is a non-zero morphism $\Serre^{-1}(M)\to U$.,
\item The functor $s^{\Serre^{-1}(M)}$ is a composition 
factor of $\Hom_\calC(-,U)$.
\end{enumerate}
\end{corollary}

\begin{proof}
  This is immediate from Corollary~\ref{composition-factor-criterion}
  and the definition of a Serre functor.
\end{proof}

Following Linckelmann~\cite{Lin} (who attributes the terminology
to Happel~\cite{Hap}) we say that the third morphism
$\gamma$ in an Auslander-Reiten triangle
$$
X\xrightarrow{\alpha} Y\xrightarrow{\beta} Z\xrightarrow{\gamma} \Sigma X
$$
is \textit{almost vanishing}. Notice that the domain and
codomain of $\gamma$ are both indecomposable in this definition.
An almost vanishing morphism determines the corresponding
Auslander-Reiten triangle by completing it to a triangle and
rotating to put it in the right position. Equally, 
an almost vanishing morphism exists with domain $Z$ 
(or codomain $X$) if and only if there is an 
Auslander-Reiten triangle with $Z$ on the 
right (or $X$ on the left -- 
since these properties are preserved by $\Sigma$).
 
As an example, when $\calC=\stmod(A)$ is the stable module
category of a symmetric algebra, $\gamma$ is almost vanishing
if and only if  it represents an almost split sequence of
$A$-modules as an element of $\Ext_A^1(W,\Omega(U))$.

Almost vanishing morphisms have been used in several places
in the literature. They underlie the construction of natural
transformations in the graded center in \cite{Lin} and \cite{LS}.
They provide a construction of ghost maps showing that Freyd's
generating hypothesis fails in general for the stable module
category $\stmod(kG)$, when $G$ is a finite group  \cite{CCM}.
They were used by Happel \cite{Hap} in constructing
Auslander-Reiten triangles in bounded derived categories
(where they exist). We present several 
characterizations of these morphisms. Most of these are 
well known, but conditions (2) and (3) may be less familiar.

\begin{proposition}
\label{almost-vanishing-characterization}
Let $\calC$ be a Hom-finite, Krull-Schmidt triangulated
category with Serre functor $\Serre$, and let $f:X\to Y$
be a morphism between indecomposable objects in $\calC$.
The following are equivalent:
\begin{enumerate}
\item The map $f$ is almost vanishing.
\item $f$ is non-zero, and for all objects $U$, $f$ factors
  through every non-zero morphism $U\to Y$.
\item $f$ is non-zero, and for all objects $V$, $f$ factors
  through every non-zero morphism $X\to V$.
\item Whenever $g:U\to X$ is
not a split epimorphism in $\calC$ then $fg=0$.
\item Whenever $h:Y\to U$ is
not a split monomorphism then $hg=0$.
\item The map
$\Hom_\calC(-,f):\Hom_\calC(-,X)\to \Hom_\calC(-,Y)$
factors through a simple functor.
\end{enumerate}
Thus morphisms $f$ satisfying any (and hence all) of the
above conditions are determined up to scalar multiple. For
such a morphism,  $Y\cong\Serre(X)$.
\end{proposition}

The word `split' is redundant in conditions (4) and (5) 
because all monomorphisms and epimorphisms in a 
triangulated category are split.

\begin{proof}
We start by observing that the implication $(1)\Rightarrow (6)$ is part 
of Proposition~\ref{simple-functors-finitely-presented}. For the 
converse $(6)\Rightarrow(1)$, if (6) holds then the image of 
$\Hom_\calC(-,f)$ must be the simple socle $s^{\Serre^{-1} Y}$ of 
$\Hom_\calC(-,Y)$, by Proposition~\ref{prop:injective}, 
and so $Y\cong\Serre(X)$, 
and $f$ is determined up to a scalar multiple. We know that there exists 
an almost vanishing map $g:X\to \Serre X$,
and it has the same property as $f$. Hence $f$ is a
scalar multiple of $g$, and $f$ is almost vanishing.

$(1)\Rightarrow (2)$. Suppose that $f$ is almost vanishing, and
  let $U\to Y$ be a non-zero morphism. Then $X\cong \Serre^{-1}(Y)$
  and $s^X\cong s^{\Serre^{-1}(Y)}$ is a composition factor of
  $\Hom_\calC(-,U)$ by Proposition~\ref{composition-factors},
  and we have morphisms between projective functors
  $\Hom_\calC(-,X)\to \Hom_\calC(-,U)\to \Hom_\calC(-,Y)$ with
  composite mapping to the simple socle of $\Hom_\calC(-,Y)$.
  This means the corresponding composite is almost vanishing
  and provides a factorization of $f$ as in (2).

$(2)\Rightarrow(1)$ Suppose $f$ satisfies (2), and let
  $\phi: \Serre^{-1}(Y)\to Y$ be almost vanishing. There is a
  factorization of $f$ as
  $X\xrightarrow{\gamma} \Serre^{-1}(Y)\xrightarrow{\phi} Y$.
  Since the image of $\Hom_\calC(-,\phi)$ is the simple top
  $s^{\Serre^{-1}(Y)}$ and
  $f\ne 0$, $\Hom_\calC(-,\gamma):\Hom_\calC(-,X)\to
    \Hom_\calC(-,\Serre^{-1}(Y))$
  maps onto the simple top and hence is surjective, by
  Nakayama's lemma. Therefore $\gamma$ is an isomorphism,
  and $f$ is almost vanishing.

The equivalence of (1) and (3) is similar. 

That (1) implies (4) follows 
because $g$ factors through $W$ in the Auslander-Reiten 
triangle $\Sigma^{-1}Y\xrightarrow{\alpha} 
W\xrightarrow{\beta} X\xrightarrow{f} Y$ 
so that $fg=f\beta g'$ for some $g'$, and 
this composite is zero because $f\beta=0$.

To get that (4) implies (1) we  complete $f$ to a triangle and 
rotate to get a triangle 
$\Sigma^{-1}Y\xrightarrow{\alpha} W\xrightarrow{\beta} X\xrightarrow{f} Y$. 
This is an Auslander-Reiten triangle because 
$f\ne 0$, $X$ and $\Sigma^{-1}Y$ are indecomposable, 
and condition (4) implies that any morphism $g:U\to X$ 
which is not a split epimorphism factors through $W$.

The equivalence $(1)\Leftrightarrow (5)$ is similar.
\end{proof}

In the next section we consider morphisms 
$f:X\to Y$ for which the image of the natural 
transformation of representable functors
$$
\Hom_\calC(-,f): \Hom_\calC(-,X)\to \Hom_\calC(-,Y)
$$ 
has finite composition length. To prepare for this we present 
some results which identify the occurrence of composition factors.

\begin{proposition}
\label{image-composition-factor}
Assume that $\calC$ has a Serre functor $\Serre$. Let
$f:X\to Y$ be a morphism between indecomposable objects,
and let $V$ be an indecomposable object. The following are equivalent. 
\begin{enumerate}
\item The functor $s^V$ is a composition factor of the image of
$$
\Hom_\calC(-,X)\xrightarrow{\Hom(-,f)}\Hom_\calC(-,Y).
$$
\item There is a morphism $V\xrightarrow{\gamma} X$ so
  that $f\gamma\ne0$.
\item There are morphisms
  $\Serre^{-1}(Y)\xrightarrow{\xi} V\xrightarrow{\gamma} X$
  so that $f\gamma\xi$ is almost vanishing.
\end{enumerate}
\end{proposition}

\begin{proof}
  Because $\Hom(-,V)$ is projective and has unique simple
  quotient $s^V$ we see that $s^V$ is a composition factor
  of the image of $\Hom(-,f)$ if and only if there is a morphism 
  $\Hom(-,\gamma):\Hom_\calC(-,V)\to\Hom_\calC(-,X)$ whose
  image is not in the kernel of $\Hom(-,f)$. By Yoneda's Lemma,
  this happens if and only if there is a morphism
  $\gamma:V\to X$ so that $f\gamma\ne0$. By
  Proposition~\ref{almost-vanishing-characterization}
  conditions (2) and (3) are equivalent.
\end{proof}

If $\alpha:F\to G$ is a natural transformation of functors
defined on $\calC$, we say that the \textit{support}
of $\alpha$ is the set of isomorphism classes of indecomposable
objects $M$ for which $\alpha_M: F(M)\to G(M)$ is non-zero.

\begin{corollary}
\label{finite-composition-length}
Assume that $\calC$ has a Serre functor $\Serre$ and let
$f:X\to Y$ be a morphism between indecomposable objects. 
The following are equivalent:
\begin{enumerate}
\item The image of $\Hom_\calC(-,f)$ has finite composition length.
\item There are only finitely many isomorphism classes of
  indecomposable modules $V$ with a morphism $\gamma:V\to X$
  so that $f\gamma\ne0$.
\item There are only finitely many isomorphism classes of 
indecomposable modules $V$ such that there are morphisms 
$\Serre^{-1}(Y)\xrightarrow{\xi} V\xrightarrow{\gamma} X$ for which 
$f\gamma\xi$ is almost vanishing.
\item The support of $\Hom_\calC(-,f)$ is finite. 
\item Whenever $\phi:\Serre^{-1}(Y)\to X$ is such that $f\phi$ 
is almost vanishing then $\phi$ can be expressed as a sum of 
composites of (finitely many) irreducible morphisms.
\end{enumerate}
Furthermore, if the image of $\Hom_\calC(-,f)$ has finite
composition length then its composition factors are the
$s^V$ for which there are morphisms
$\Serre^{-1}(Y)\xrightarrow{\xi} V\xrightarrow{\gamma} X$, both
of which are finite composites of irreducible morphisms, and so that
$f\gamma\xi$ is almost vanishing.
\end{corollary}

\begin{proof} The equivalence of the first four statements
  is immediate from Proposition~\ref{image-composition-factor}.

  $(1)\Leftrightarrow(5)$ The image has finite composition length
  if and only if $\Rad_\calC^n(-,X)$ is contained in the kernel of
  $\Hom_\calC(-,f)$ for some $n$. Thus, assuming (1), no morphism
  $\gamma:Y\to X$ with $f\gamma\ne 0$ lies in $\Rad_\calC^n(Y,X)$,
  and so such a morphism cannot be expressed as a sum of composites
  of $n$ or more irreducible morphisms. Thus (5) holds. Conversely,
  assume (5) holds. We may take an almost vanishing morphism
  $\phi:\Serre^{-1}(Y)\to X$ and express it as a sum of composites
  of irreducible morphisms, deducing that $\phi$ does not lie in
  $\Rad_\calC^n(\Serre^{-1}(Y),X)$ for some $n$. Since for each non-zero
  morphism $\gamma:Y\to X$ there is a factorization $\phi=\gamma \xi$
  for some $\xi$, we deduce that every non-zero morphism $\gamma$
  lies outside $\Rad_\calC^n(Y,X)$. This shows that $\Rad_\calC^n(-,X)$
  is contained in the kernel of $\Hom_\calC(-,f)$ and so (1) holds.

  The composition factors are as claimed because firstly, by
  Proposition~\ref{image-composition-factor}, the $s^V$ for which
  $f\gamma\xi$ is as described are among the composition factors.
  The complete set of composition factor arises without the
  requirement that $\xi$ and $\gamma$ be finite composites of
  irreducible morphisms, but we see from (5) that they must be
  sums of composites of irreducible morphisms. The $V$ which
  can arise from sums of composites of irreducible morphisms
  are the same as the $V$ which arise from composites of
  irreducible morphisms.
\end{proof}

\section{Elements of the graded center with finite support}

In the context of stable module categories $\stmod(A)$ for
symmetric algebras $A$, Linckelmann \cite{Lin} constructed
certain elements of the graded center of $\stmod(A)$ of
degree $-1$. In that situation the shift is given by
$\Sigma=\Omega^{-1}$, the inverse of the Heller operator
and the Serre functor is $\Serre=\Omega$. For each finitely
generated indecomposable non-projective module $U$, he
constructed a natural transformation
$\zeta: \Id_\calC \to \Omega$ such that $\zeta_U:U\to\Omega(U)$
is almost vanishing (i.e. represents an almost split
sequence ending in $U$), and such that $\zeta(V)=0$ for
any finitely generated indecomposable non-projective module
$V$ which is not isomorphic to $\Omega^n(U)$, for any
integer $n$.  Linckelmann and Stancu~\cite{LS} then combined
this construction with the existence of periodic modules of period one to
produce elements of  the graded center in degree 0.

We assume throughout that $\calC$ is a
Hom-finite, Krull-Schmidt, $k$-linear triangulated  category
with Serre functor $\Serre$ and Auslander-Reiten translate
$\tau M= \Sigma^{-1}\Serre(M)$. 
Our first goal is to show that the natural transformations
of the kind constructed by Linckelmann are the only ones with small support.

\begin{proposition}
\label{single-object-support}
Let $F:\calC\to\calC$ be a $k$-linear endofunctor and
suppose $\alpha: \Id_\calC\to F$ is a natural transformation
with support consisting of a single indecomposable object $M$.
Then $F(M)=\Serre(M)$ and $\alpha_M:M\to \Serre(M)$ is an
almost vanishing morphism.
\end{proposition}

\begin{proof}
Let $\alpha$ have support only on $M$. Consider the image of 
$$
\Hom_\calC(-,M)\xrightarrow{\Hom_\calC(-,\alpha_M)} \Hom_\calC(-,F(M))
$$
as a subfunctor of $\Hom_\calC(-,F(M))$.
If the image only has one composition factor (which must be $s^M$,
the simple top of $\Hom_\calC(-,M)$) this composition factor must
be the socle of $\Hom_\calC(-,F(M))$. Since the socle is
$s^{\Serre^{-1}(F(M))}$ we deduce that $M=\Serre^{-1}(F(M))$, so
$F(M)=\Serre(M)$, and that $\alpha_M$ is almost vanishing
by Proposition~\ref{almost-vanishing-characterization}. 

If the image has another composition factor $s^X$ for some
object $X$, this also appears as a composition factor of
$\Hom_\calC(-,M)$. Hence, by projectivity of the representable
functor $\Hom_\calC(-,X)$, there is a non-zero morphism
$\Hom_\calC(-,X)\to \Hom_\calC(-,M)$ so that the composite
$$
\Hom_\calC(-,X)\to \Hom_\calC(-,M)\xrightarrow{\Hom_\calC(-,\alpha_M)}
\Hom_\calC(-,F(M))
$$
is non-zero. By Yoneda's Lemma this corresponds to a
homomorphism $\phi:X\to M$ so that $\alpha_M\circ \phi\ne 0$.
Since $\alpha$ is a natural transformation it follows that
$F(\phi)\circ\alpha_X\ne 0$, so that $\alpha$ has $X$ in its
support, contradicting our hypothesis. Hence, the image has
only one composition factor, a case we have already considered.
\end{proof}

We apply this to the situation considered by Linckelmann
and Stancu~\cite{LS}, namely the stable module category
$\stmod(kG)$ of a finite $p$-group $G$ over an algebraically
closed field of characteristic $p$. We show that the
elements of the graded center that they construct in positive 
degree are the only ones with support a single module.

\begin{corollary}
  Let $G$ be a finite group, $k$ a field and $r$ an integer.
  Let $\phi: \Id_{\stmod(kG)}\to\Sigma^r$ be a natural
  transformation with support a single indecomposable
  module $\{ M\}$. Then $\Sigma^r\cong\Omega(M)$ and 
$\phi_M:M\to\Omega(M)$ is an almost vanishing morphism. 
  Thus if $r\ne 1$, then $M\cong \Omega^{r+1}(M)$ is a
  periodic module and $\phi$  is, up to scalar multiple,
  one of the natural transformations constructed by
  Linckelmann and Stancu. 
\end{corollary}

\begin{proof}
  In $\stmod(kG)$ we have $\Serre=\Omega=\Sigma^{-1}$.
  According to Proposition~\ref{single-object-support},
  $\Sigma^r M=\Serre(M)$ and $M\to \Serre(M)$ is almost vanishing. 
The remaining assertions are immediate.
\end{proof}

Linckelmann and Stancu were interested in elements of the
graded center of $\stmod(kG)$. The degree $n$ elements of
this ring are the natural transformations $\Id_\calC\to\Sigma^n$
which commute (in a sense which includes a sign) with $\Sigma$.
Such natural transformations have support consisting
of a single indecomposable object only if that object is
periodic under $\Sigma$. Allowing such transformations to
have support on a single $\Sigma$-orbit of indecomposable
objects, which might not be periodic, we obtain the following.

\begin{proposition}
\label{single-sigma-orbit}
Let $M$ be an indecomposable object of $\calC$ for which
there are no irreducible maps $\Sigma^rM\to M$ for any
$r\in\ZZ$, and let $F:\calC\to\calC$ be a $k$-linear endofunctor.
Suppose that $\alpha: \Id_\calC\to F$ is a natural
transformation whose support is contained in
$\{\Sigma^rM \bigm| r\in\ZZ\}$. Then $F(M)=\Serre(M)$, and
for each $r$, the map $\alpha_{\Sigma^r M} : \Sigma^r M\to F(\Sigma^rM)$
is almost vanishing. Thus $\alpha$ is one of the natural
transformations constructed by Linckelmann in \cite{Lin}.
\end{proposition}

The hypothesis that there are no irreducible maps
$\Sigma^rM\to M$ for any $r\in\ZZ$ holds in many cases
of interest. For example, it  always holds if $M$
belongs to an Auslander-Reiten quiver component of tree
class $A_\infty$. By \cite{Web} it can be seen to hold
most of the time for $\stmod(kG)$ when $G$ is a finite group.

\begin{proof} As in the proof of 
Proposition~\ref{single-object-support}, consider the image of 
$$
\Hom_\calC(-,M)\xrightarrow{\Hom_\calC(-,\alpha_M)} \Hom_\calC(-,F(M)).
$$
This has $s^M$ as a composition factor, and if it
has more composition factors than this it must have
one of the form $s^E$ where $E\to M$ is an irreducible
morphism since such simple functors form the second
radical layer of the projective cover of $s^M$. This
would mean $E$ does not have the form $\Sigma^r M$ and
that $\alpha$ has $E$ in its support, which is not
possible. We conclude that the image is the simple
functor $s^M$, and as before, $FM=\Serre(M)$ and $\alpha_M$
is an almost vanishing morphism.
\end{proof}

We now consider elements of the graded center of $\calC$
with support larger than a single $\Sigma$-orbit of objects.
One way to construct such elements is to add two elements
which have support on different shift orbits: the resulting
natural transformation $\alpha:\Id_\calC\to\Sigma^r$ 
has the property that for every morphism $f:M\to N$
between indecomposable objects in different shift orbits
we have $\alpha_Nf=0$. We consider $\alpha$ with
$\alpha_Nf\ne 0$ for some non-isomorphism $f$. This is
equivalent to requiring that the support of the natural
transformation $\Hom_\calC(-,\alpha_N)$ has size at least 2
for some $N$.

We recall that a triangulated category $\calC$ is
$d$-Calabi-Yau if $\Sigma^d$ is a Serre functor. It follows
from \cite{RvdB} that such a category has Auslander-Reiten
triangles. In the next result we refer to the
Auslander-Reiten quiver simply as the `quiver'. We recall 
that the term \textit{mesh} denotes a region of this quiver 
bounded by the objects which appear in the three left terms 
of an Auslander-Reiten triangle \cite{BG}.

\begin{theorem}
\label{finite-comp-length-theorem}
Let $\calC$ be a $k$-linear, Hom-finite, Krull-Schmidt
triangulated category. Let $\alpha:\Id_\calC\to \Sigma^r$
be a natural transformation in the graded center of $\calC$.
Fix an indecomposable object $N$ of $\calC$.
We suppose that
\begin{enumerate}
\item $\calC$ is a $d$-Calabi-Yau category for some integer $d$, 
\item for all objects $U$ in the quiver component
of $N$, $\Sigma^{r-d} U$ and $U$ lie in the same $\tau$-orbit,
\item every mesh in the quiver component of $N$ has at most
  two middle terms,  and
\item  for all objects $U$ in the same quiver component as $N$,
  the support of the natural transformation
$\Hom_\calC(-,\alpha_U): \Hom_\calC(-,U)\to \Hom_\calC(-,\Sigma^r U)$
is finite, and for some $U$ it has size at least 2.
\end{enumerate}
Then the support of $\alpha$ contains the entire quiver component of $N$.
\end{theorem}

\begin{proof}
Let $U$ be an indecomposable object in the quiver component of $N$ for 
which the support of $\Hom_\calC(-,\alpha_U)$ has size at least 2. i
Since $\Hom_\calC(-,\alpha_U)$ has finite composition length,
  by Proposition~\ref{almost-vanishing-characterization} and 
Corollary~\ref{finite-composition-length} there is a
  morphism $\phi: \Sigma^{r-d}U\to U$, which is a sum of finite composites of
  irreducible morphisms, such that $\alpha_U\phi$ is almost vanishing.
  Because the support of $\Hom_\calC(-,\alpha_U)$ has size at least 2,
  $\phi$ is not an isomorphism.

  We claim that the composite of morphisms in any path in the
  Auslander-Reiten quiver from $\Sigma^{r-d}U$ to $U$ also has the
  same property as $\phi$, and in fact equals $\pm\phi$. This is
  because whenever we have a pair of consecutive irreducible
  morphisms in such a path of the form $\tau V \to W\to V$ the
  Auslander-Reiten triangle $\tau V\to E\to V\to \Sigma\tau V$
  has middle term $E$ with at most two indecomposable summands,
  one of which is $W$. If $E=W\oplus X$ for some $X$ we can
  replace the maps into and out of $W$ by irreducible morphisms
  $\tau V\to X\to V$, because the composite
  $\tau V\to W\oplus X\to V$ is zero, so that the new irreducible 
morphisms have composite $(-1)$ times the composite of the old. 
Repeating this operation
allows us to move from any path from $\Sigma^{r-d}U$ to $U$
to any other path, changing the composite by $(-1)$ each time. 
Now $\phi$ must be a linear combination of composites along these 
paths, but since the composites are all the same up to sign we 
deduce that $\phi$ could be taken to be the composite of the 
irreducibles along any of the paths.

Since $\Sigma^{r-d}U$ and $U$ lie in the same $\tau$-orbit, there 
is a path in the quiver from $\Sigma^{r-d}U$ to $U$ going through 
each member of the $\tau$-orbit of $U$ between these two objects. 
We deduce that for every irreducible morphism with codomain $U$ the 
domain of that morphism lies in the support of $\alpha$. This and the fact
  that $\alpha$ commutes (up to sign) with $\Sigma$, and hence with $\tau$,
  implies that all objects in the component of $N$ lie in
  the support of $\alpha$.
\end{proof}

In the next section we present an example of a natural
transformation satisfying the conditions of
Theorem~\ref{finite-comp-length-theorem} in the context
of the stable module category of a group with a dihedral 
Sylow 2-subgroup in characteristic 2.
In general it is not always possible to find such examples, as we now see. 

\begin{corollary}
\label{A-corollary}
With the same hypotheses as in
Theorem~\ref{finite-comp-length-theorem}, suppose further that
the quiver component containing $N$ has type $A_\infty$.
Then no such natural transformation $\alpha$ can exist.
\end{corollary}

\begin{proof} Suppose there were such a natural transformation $\alpha$.
Its support would have to contain the quiver component containing $N$, 
and for any choice of indecomposable object $N_0$ in this quiver 
component the proof of Theorem~\ref{finite-comp-length-theorem} shows 
that there is an irreducible morphism $f:U\to N_0$ with 
$\alpha_{N_0}f\ne 0$. We may choose $N_0$ so that it lies on 
the rim of the quiver, as in the following diagram. 
{
\def\tempbaselines
{\baselineskip16pt\lineskip3pt
   \lineskiplimit3pt}
\def
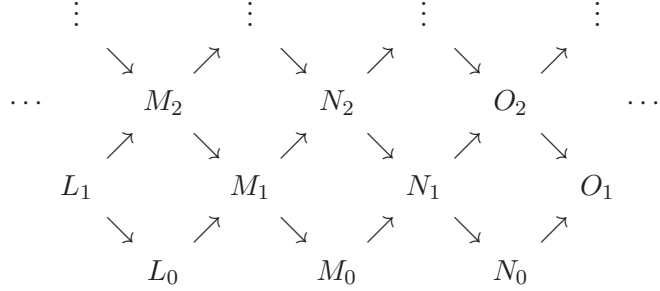
\diagram#1{\null\,\vcenter{\tempbaselines
\mathsurround=0pt
    \ialign{\hfil$##$\hfil&&\quad\hfil$##$\hfil\crcr
      \mathstrut\crcr\noalign{\kern-\baselineskip}
  #1\crcr\mathstrut\crcr\noalign{\kern-\baselineskip}}}\,}

\def\clap#1{\hbox to 0pt{\hss$#1$\hss}}
$$
\diagram{&\clap{\vdots}&&&&\clap{\vdots}&&&&\clap{\vdots}&&&&\clap{\vdots}\cr
&&\searrow&&\nearrow&&\searrow&&\nearrow&&\searrow&&\nearrow&\cr
\cdots&&&\clap{M_2}&&&&\clap{N_2}&&&&\clap{O_2}&&&\cdots\cr
&&\nearrow&&\searrow&&\nearrow&&\searrow&&\nearrow&&\searrow&\cr
&\clap{L_1}&&&&\clap{M_1}&&&&\clap{N_1}&&&&\clap{O_1}\cr
&&\searrow&&\nearrow&&\searrow&&\nearrow&&\searrow&&\nearrow&\cr
&&&\clap{L_0}&&&&\clap{M_0}&&&&\clap{N_0}\cr
}
$$
}
There is no path of irreducible morphisms from $\Sigma^{r-d} N_0$
to $N_0$ with non-zero composite unless $r=d$. This is because 
$\Sigma^{r-d} N_0$ is also on the rim, and such a
path has composite equal to that of a path which has two
irreducible morphisms between consecutive objects on the
rim, and the composition of these morphisms is zero. Such a 
path was necessary to the existence of $\alpha$ in the proof 
of Theorem~\ref{finite-comp-length-theorem}, so this situation 
cannot occur. When $r=d$  the support
of $\Hom_\calC(-,\alpha_{N_0})$ has size 1, because there is no
finite chain of irreducible morphisms from $N_0$ to $N_0$ other
than the empty chain at $N_0$. This shows that $\alpha_{N_0}$ is almost
vanishing, so that $\alpha_{N_0}f=0$, a contradiction. Hence i
no  $\alpha$ can exist as in Theorem~\ref{finite-comp-length-theorem}.
\end{proof}

\begin{corollary}
  Let $\calC=\stmod(B)$ be the stable module category of a
block with wild representation type of a group algebra $kG$. 
Let $\alpha$ be an element of the graded center of $\calC$.
\begin{enumerate}
\item If $\alpha$ is supported on only finitely many 
$\tau$-orbits then $\alpha$ is a sum of elements which are supported on 
  single $\tau$-orbits, each of which is of the kind described 
in Proposition~\ref{single-sigma-orbit}. Thus $\alpha_Yf=0$ 
for every non-isomorphism $f:X\to Y$ between indecomposable objects.
\item If there is any non-isomorphism $f:X\to Y$ between 
indecomposable objects so that $\alpha_Yf\ne0$ then such 
an $f$ can be found which  is not a finite composite of irreducible 
morphisms. In this case $\alpha$ is not supported on only finitely 
many $\tau$-orbits.
\end{enumerate}
\end{corollary}

\begin{proof}
  We exploit the fact, using a theorem of Erdmann \cite{Erd},
  that all quiver components of $\calC$ have type $A_\infty$ and satisfy
  conditions (1), (2) and (3) of
  Theorem~\ref{finite-comp-length-theorem}.  
  
  To prove (1), if $\alpha$ were supported on only finitely many 
$\tau$-orbits then the support of $\Hom_\calC(-,\alpha_Y)$ would 
be finite for all indecomposable $Y$ and by Corollary ~\ref{A-corollary} 
such $\alpha$ cannot exist unless this support has size 1 for 
every $Y$. This is equivalent to requiring that $\alpha_Yf=0$ 
for every non-isomorphism $f:X\to Y$ between indecomposable objects, 
and that $\alpha$ is a sum of natural transformations supported 
on single $\tau$-orbits.
  
With the hypothesis of (2), we must have that some
  $\Hom_\calC(-,\alpha_Y)$ has infinite support.
Finding $f:X\to Y$ so that $\alpha_Yf$ is almost vanishing as 
in Proposition~\ref{almost-vanishing-characterization}, 
we find by Corollary~\ref{finite-composition-length} that $f$ is 
not a composite of irreducible morphisms.
\end{proof}

\begin{remark} \label{rem:finite-compose}
In the case of modules in a block of wild type in a group algebra, it 
seems likely that any map $f: X \to Y$ as above, 
that is not a composite of a finite
number of irreducible maps, should factor through a module that is not
in the quiver component of $X$, implying that $\alpha$ would have 
support on more than one quiver component. This is easily verified in some 
specific cases, but seems difficult to prove in general. 
\end{remark}

\section{An example: groups with dihedral Sylow 2-subgroups}

In this section we show, under certain circumstances, that there exist 
natural transformations in the graded center 
of the stable module category that are supported
on only a single component of the Auslander-Reiten quiver,
and which are not sums of the natural transformations constructed
by Linckelmann in \cite{Lin}. Furthermore, our natural
transformations satisfy the finiteness condition of
Theorem~\ref{finite-comp-length-theorem}. 

We assume throughout that $k$ is an algebraically closed field
of characteristic $2$ and that $G$ is a finite group 
with a dihedral Sylow 2-subgroup having order at least 8. The
group algebra $kG$  has tame
representation type, and a primary fact in the example is that the 
Auslander-Reiten quiver component which contains the trivial module has
tree class $A_\infty^\infty$  \cite{Web} and consists entirely
of endotrivial modules
(see \cite{AC}). By definition,
a $kG$-module $M$ is endotrivial provided $\Hom_{k}(M, M) \cong k \oplus P$ 
where $P$ is a projective $kG$-module. We note that a $kG$-module is
endotrivial if and only if its restriction to every elementary abelian
$p$-subgroup is endotrivial and that the tensor product of two endotrivial
modules is again an endotrivial module (See \cite{CT}).

Suppose that $S$ is a Sylow 2-subgroup of 
$G$ and that $E_1$ and $E_2$ are representatives of the two conjugacy 
classes of elementary abelian subgroups of order 4 in $S$.  
The Auslander-Reiten quiver 
containing the trivial $kG$-module has the form 
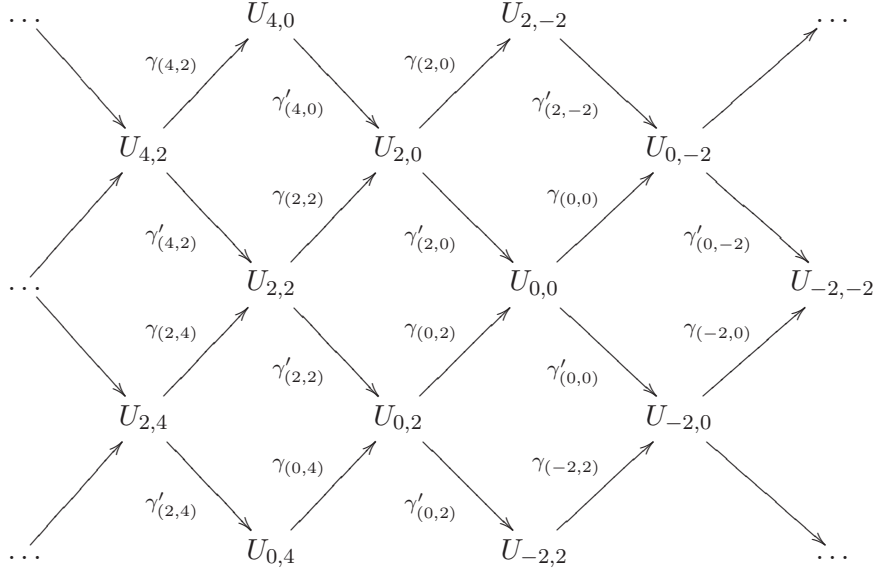
\[
\xymatrix@+.8pc@C-.9pc{
\dots \ar[dr] && U_{4,0} \ar[dr]_{\gamma_{(4,0)}^\prime}  && 
U_{2,-2} \ar[dr]_{\gamma_{(2,-2)}^\prime} && \dots \\
& U_{4,2} \ar[ur]^{\gamma_{(4,2)}} \ar[dr]_{\gamma_{(4,2)}^\prime} && 
U_{2,0} \ar[ur]^{\gamma_{(2,0)}} \ar[dr]_{\gamma_{(2,0)}^\prime} &&
U_{0,-2} \ar[dr]_{\gamma_{(0,-2)}^\prime} \ar[ur] && \\
\dots \ar[ur] \ar[dr] && U_{2,2} \ar[ur]^{\gamma_{(2,2)}} 
\ar[dr]_{\gamma_{(2,2)}^\prime} && 
U_{0,0} \ar[ur]^{\gamma_{(0,0)}} \ar[dr]_{\gamma_{(0,0)}^\prime} && 
U_{-2,-2} && \\
& U_{2,4} \ar[ur]^{\gamma_{(2,4)}} \ar[dr]_{\gamma_{(2,4)}^\prime} && 
U_{0,2} \ar[ur]^{\gamma_{(0,2)}} \ar[dr]_{\gamma_{(0,2)}^\prime} &&
U_{-2,0} \ar[ur]^{\gamma_{(-2,0)}} \ar[dr] && \\
\dots \ar[ur] && U_{0,4} \ar[ur]^{\gamma_{(0,4)}} && 
U_{-2,2} \ar[ur]^{\gamma_{(-2,2)}} && \dots \\
}
\]
where $U_{0,0} \cong k$, $U_{i,i} \cong \Omega^i(k)$ and 
$U_{i,j}$ is an endotrivial module with the property that 
$U_{i,j}\downarrow_{E_1} \cong \Omega^i(k_{E_1})$ and
$U_{i,j}\downarrow_{E_2} \cong \Omega^j(k_{E_2})$ (See \cite{AC, Web}).
The almost split sequence ending in the trivial module has the 
form 
\[
\xymatrix{
0 \ar[r] & \Omega^2(k) \ar[r] & U_{2,0} \oplus U_{0,2} \ar[r] 
& k \ar[r] & 0
}
\]
Note that all modules in the component of the trivial module have
odd dimension since they are endotrivial. 
If $M$ is an indecomposable $kG$-module of odd 
dimension, then by \cite{AC} the almost split sequence ending
in $M$ is (modulo 
projective summands) 
\[
\xymatrix{
0 \ar[r] & \Omega^2(k) \otimes M \ar[r] & U_{2,0} \ \otimes M \ \oplus 
U_{0,2} \otimes M \ar[r] & M \ar[r] & 0
}
\]
In particular, we see that $U_{i,j} \otimes U_{s,t} \cong 
U_{i+s, j+t} \oplus P$ for some projective module $P$. 

\begin{lemma} \label{lemma:exact}
For any $n$ there is an exact sequence having the form
\[
\xymatrix{
\calE_n:  
& 0 \ar[r] & \Omega^{2n}(k) \ar[rr]^{
\begin{pmatrix} \alpha, \beta \end{pmatrix} \quad} 
&& U_{2n,0} \oplus U_{0,2n} 
\ar[r]^{\qquad \begin{pmatrix} \gamma \\ \delta\end{pmatrix}} 
& k \ar[r] & 0
}
\]
where $\alpha = \gamma_{(2n,2)} \dots  \gamma_{(2n,2n-2)}  
\gamma_{(2n,2n)}$, 
$\beta = \gamma_{(2,2n)}^\prime \dots 
\gamma_{(2n-2,2n)}^\prime   \gamma_{(2n,2n)}^\prime$, etc. 
That is, each map is the obvious composition of irreducible maps in 
the Auslander-Reiten quiver. 
\end{lemma}

\begin{proof}
The modules in the sequence are positioned in the Auslander-Reiten 
quiver as the vertices of a diamond. By an argument similar to the 
one used to prove Theorem~\ref{finite-comp-length-theorem} we see 
that the two composites of irreducible morphisms from 
$\Omega^{2n}(k)$ to $k$, obtained by going round the two sides 
of the diamond, are equal of opposite sign. This shows that the 
composite of the two middle morphisms in the sequence is zero. 
We will show that $U_{2n,0}\oplus U_{0,2n}\to k$ is  
surjective and that the left side of the sequence 
is the kernel of this surjection. In what follows, we write
$\Rad^n_{kG}$ for the $n^{th}$ radical of $\Hom_{kG}$.
 
We use functorial methods to establish  this. 
We claim that for any module $M$ in a stable component of the 
Auslander-Reiten quiver of ${kG}$-modules of type 
$A_\infty^\infty$, the composition factors of 
$\Hom_{kG}(-,M)/\Rad^\infty_{kG}(-,M)$ are the $s^V$ for 
which there is a path of irreducible morphisms from $V$ to $M$. 
More specifically, the composition factors of 
$\Rad_{kG}^n(-,M)/\Rad_{kG}^{n+1}(-,M)$ are the $s^V$ for which 
there is a path of $n$ irreducible morphisms $V$ to $M$, each $s^V$  
taken with multiplicity 1. This may be proved by considering the 
projective resolutions of simple functors, such as
$$
\begin{aligned}
0\to\Hom_{kG}(-,\tau M)\to\Hom_{kG}(-, L_1)&\oplus\Hom_{kG}(-,L_2)\cr
&\to\Hom_{kG}(-, M)\to s^M\to 0\cr
\end{aligned}
$$
where $0\to\tau M\to L_1\oplus L_2\to M\to 0$ is an almost split sequence.
For each $n\ge 2$ this restricts to an exact sequence
$$
\begin{aligned}
0\to\Rad^{n-2}_{kG}(-,\tau M)\to\Rad^{n-1}_{kG}(-, L_1)&\oplus
\Rad^{n-1}_{kG}(-,L_2)\cr
&\to\Rad^{n}_{kG}(-, M)\to 0\cr
\end{aligned}
$$
since the morphisms are obtained by composition with an irreducible 
morphism. Hence we obtain for each $n\ge 1$ an exact sequence
$$
\begin{aligned}
0&\to\Rad^{n-2}_{kG}(-,\tau M)/\Rad^{n-1}_{kG}(-,\tau M)\cr
&\to\Rad^{n-1}_{kG}(-, L_1)/\Rad^{n}_{kG}(-, L_1)
\oplus\Rad^{n-1}_{kG}(-,L_2)/\Rad^{n}_{kG}(-,L_2)\cr
&\to\Rad^{n-1}_{kG}(-, M)/\Rad^{n}_{kG}(-, M)\to 0\cr
\end{aligned}
$$
where we take $\Rad^{-1}=\Rad^0$. We also know that the composition factors of
$$
\Rad^{n-2}_{kG}(-,\tau M)/\Rad^{n-1}_{kG}(-,\tau M)
$$
are the composition factors of
$$
\Rad^{n-2}_{kG}(-, M)/\Rad^{n-1}_{kG}(-, M)
$$
with $\tau$ applied and that each indecomposable representable 
functor has a simple top. This provides a system of equations 
which allows us to compute the composition factors by recurrence: 
in a Grothendieck group,
$$
\begin{aligned}
\Rad^0_{kG}(-,M)/\Rad^1_{kG}(-,M)=& s^M\cr
\Rad^1_{kG}(-,M)/\Rad^2_{kG}(-,M)=&
\Rad^0_{kG}(-,L_1)/\Rad^1_{kG}(-,L_1)\cr
&+
\Rad^0_{kG}(-,L_2)/\Rad^1_{kG}(-,L_2)\cr
=&s^{L_1}+s^{L_2}\cr
\Rad^2_{kG}(-,M)/\Rad^3_{kG}(-,M)=&
\Rad^1_{kG}(-,L_1)/\Rad^2_{kG}(-,L_1)\cr
&+
\Rad^1_{kG}(-,L_2)/\Rad^2_{kG}(-,L_2)\cr
&-\Rad^0_{kG}(-,\tau M)/\Rad^1_{kG}(-,\tau M)\cr
=&s^{L_{11}}+s^{\tau M}+s^{L_{22}}+s^{\tau M} - s^{\tau M}\cr
=&s^{L_{11}}+s^{\tau M}+s^{L_{22}}\cr
\end{aligned}
$$
where $0\to\tau L_i\to L_{ii}\oplus\tau M\to L_i\to 0$, $i=1,2$ 
are almost split sequences; and so on. We conclude that each irreducible 
morphism, such as $L_1\to M$, induces a monomorphism
$$
\Hom_{kG}(-,L_1)/\Rad_{kG}^\infty(-,L_1)\to 
\Hom_{kG}(-,M)/\Rad_{kG}^\infty(-,M).
$$
Hence, so does every composite of irreducible morphisms induce 
such a monomorphism. By counting composition factors we see that
$$
\begin{aligned}
0&\to \Hom_{kG}(-,\Omega^{2n}(k))/\Rad_{kG}^\infty(-,\Omega^{2n}(k))\cr
&\to
\Hom_{kG}(-,U_{2n,0}\oplus U_{0,2n})/\Rad_{kG}^\infty(-,U_{2n,0}
\oplus U_{0,2n})\nonumber\cr
&\to \Hom_{kG}(-,k)/\Rad_{kG}^\infty(-,k)\nonumber\cr
\end{aligned}
$$
is exact (and the last cokernel has composition factors inside 
the diamond we are considering in the quiver).

We may now deduce that the morphism $U_{2n,0}\oplus U_{0,2n}\to k$
is surjective, because it induces a non-zero map of representable
functors and hence must be non-zero, to a module of dimension 1.
Let $K$ be the kernel of this morphism. 
Thus $0\to K\to U_{2n,0}\oplus U_{0,2n}\to k\to 0$ is 
exact and our task is to show that $K$ is $\Omega^{2n}(k)$. Then 
$$
0\to \Hom_{kG}(-,K)\to \Hom_{kG}(-,U_{2n,0}\oplus U_{0,2n})\to \Hom_{kG}(-,k)
$$
is exact (by left exactness of $\Hom$), hence so is the similar 
sequence we get after factoring out $\Rad^\infty$ from each term. 
Since the composite
$$
\Omega^{2n}(k)\to  U_{2n,0}\oplus U_{0,2n}\to k
$$
is zero we get a morphism $\Omega^{2n}(k)\to K$ (by the 
universal property of the kernel). This passes to a map 
$$
\Hom_{kG}(-,\Omega^{2n}(k))/\Rad_{kG}^\infty(-,\Omega^{2n}(k))
\to\Hom_{kG}(-,K)/\Rad_{kG}^\infty(-,K)
$$
which is an isomorphism since both terms act as the kernel in the 
sequences of $\Rad^\infty$ quotients. It follows from this that the 
irreducible morphisms to $\Omega^{2n}(k)$ and to (the summands of) $K$ 
are the same, so that the summands of $\Omega^{2n}(k)$ and of $K$ 
occupy the same positions in the Auslander-Reiten quiver. Thus $K$ 
is indecomposable, and the map  $\Omega^{2n}(k)\to K$ is an isomorphism. 
We deduce that the sequence
$$
0\to \Omega^{2n}(k)\to  U_{2n,0}\oplus U_{0,2n}\to k\to 0
$$
is exact.
\end{proof}

We notice that the sequence $\calE_n$ represents an element 
\[
\mu_n \in \Ext_{kG}(k, \Omega^{2n}(k)) \cong 
\uHom_{kG}(k, \Omega^{2n-1}(k)) \cong\hatH^{1-2n}(G,k).
\]
It is not necessary for our development, but perhaps 
interesting to note that,
considered as an element in Tate cohomology $\hatH^{1-2n}(G,k)$,
$\mu_n$  is perpendicular (under Tate duality) to the 
subspace of $\hatH^{2n-2}(G,k)$ spanned by the 
transfers from the proper $2$-subgroups of the Sylow subgroup 
of $G$. 

The important thing is that multiplication by 
$\mu_n$ induces a natural transformation from the identity functor
to $\Omega^{2n-1}$ in the stable category $\stmod(kG)$. 
That is, we first chose a cocycle $\mu_n: k \to \Omega^{2n-1}(k)$
representing $\mu_n$. The class of the cocycle as a map in the
stable category is unique. Then for any $M$ we have
a composition map $\mu_{n,M}$ given by 
\[
\xymatrix{
M \ar[r]^{\cong} & k \otimes M \quad \ar[r]^{\mu_n \otimes 1 \quad} & 
\quad \Omega^{2n-1}(k) \otimes M
\ar[r] & \Omega^{2n-1}(M) 
}
\]
where the first map sends $m$ to $1 \otimes m$, and the last is
the isomorphism in the stable category.
This is well defined in the stable category and does not depend on the 
choice of a cocycle representing $\mu_n$ or the choice of a splitting 
$\Omega^{2n-1}(k) \otimes M \cong \Omega^{2n-1}(M) \oplus P$ for some
projective module $P$. Thus, we see that $\mu_{n, -}$ is an
element of the graded center of the stable category of $kG$-modules. 

Next we note the following relevant fact. 

\begin{proposition} \label{prop:vanish}
Suppose that $\phi:M \to N$ is a homomorphism of indecomposable 
$kG$-modules such
that $M$ and $N$ do not lie in the same component 
of the Auslander-Reiten quiver.
Then $\mu_{n,N}\phi = 0$, and $\Omega^{2n}(\phi)\mu_{n,M} = 0$ 
in the stable category $\stmod(kG)$.
\end{proposition}

\begin{proof}
There is an isomorphism $\uHom_{kG}(M,N) \cong
  \uHom_{kG}(M \otimes N^*, k)$ that is natural in both $M$ and $N$. Hence,
letting $X= M \otimes N^*$, and $\theta: X \to K$ be the homomorphism 
corresponding to $\phi$, it is only necessary 
to show that $\mu_n\theta = 0$ in 
$\uHom_{kG}(X, \Omega^{2n-1}(k)) \cong \Ext_{kG}(X, \Omega^{2n}(k))$. 
That is, we need to prove that the map $\theta$ factors through the 
middle term of the sequence in the diagram:
\[
\xymatrix{
&&&& X \ar[d]^\theta \ar[dl]^\sigma \\
\calE_n: \quad 
& 0 \ar[r] & \Omega^2n(k) \ar[r]
& U_{2n,0} \oplus U_{0,2n} \ar[r]
& k \ar[r] & 0.
}
\]
In other words, it must be shown that there exist two maps 
$\sigma_1: X \to U_{2n,0}$ and $\sigma_2: X \to U_{0,2n}$
such that 
\[
\theta = \sigma_1\gamma_{(2,0)}^\prime \dots  \gamma_{(2n-2,0)}^\prime  
\gamma_{(2n,0)}^\prime + \sigma_2 \gamma_{(0,2)} \dots 
\gamma_{(0,2n-2)}  \gamma_{(0,2n)}
\]

Observe that by our hypotheses, no indecomposable direct summand
$Y$ of $X$ is in 
the Auslander-Reiten component of the trivial 
module $k$. If it were otherwise, then $Y$ would have odd 
dimension, implying that $M$ and $N$ would also have odd dimension
\cite{BC}. 
Moreover, $Y$ would be an endotrivial module, requiring that
$Y \otimes N$ have only a single non-projective summand.
However, $k$ is a direct summand of $N^* \otimes N$, and hence
$M$ must be the unique non-projective direct summand of $Y \otimes N$.
Recall from \cite{AC} that the Auslander-Reiten 
component of $N$ consist of the
non-projective summands of $Y \otimes N$ for $Y$ in
the Auslander-Reiten
component of $k$. 
Thus we would have that $M$ is in the same
Auslander-Reiten component as $N$,
contradicting our hypotheses. 

Because the row in the diagram
\[
\xymatrix{
&&& X \ar[d]^\theta \ar[dl] \\
0 \ar[r] & \Omega^2(k) \ar[r]
& U_{2,0} \oplus U_{0,2} \ar[r]
& k \ar[r] & 0.
}
\]
is an almost split sequence, there are maps $\mu_1:X \to U_{2,0}$ and
$\mu_2: X \to U_{0,2}$ such that $\theta = \gamma_{(2,0)}\mu_1
+ \gamma_{(0,2)}\mu_2$. We can iterate this process. That is, in 
the next iteration, we write $\mu_1 = \gamma_{(4,0)}\nu_1 +
\gamma_{(2,2)}\nu_2$ for $\nu_1: X \to U_{4,0}$ 
and $\nu_2: X \to U_{2,2}$ using the fact that 
$0 \to U_{4,2} \to U_{4,0} \oplus U_{2,2} \to U_{2,0} \to 0$
is an almost split sequence. 

In this way, for some $m > n$,  we write $\theta$ as a sum of maps
of the form $\zeta\sigma$ where $\sigma: X \to U_{2m-2j, 2j}$ and 
$\zeta: U_{2m-2j, 2j} \to k$ is a composition of irreducible maps 
and $0 \leq j \leq m$. Next we note that 
$\gamma_{(2i,2j-2)}^\prime\gamma_{(2i, 2j)} = 
\gamma_{(2i-2,2j)}\gamma_{(2i, 2j)}^\prime$. Thus, since $m >n$, the 
map $\zeta$ factors either through $U_{2n, 0}$
or through  $U_{0,2n}$. It follows that $\theta$ factors through
$U_{2n,0} \oplus U_{0,2n} \to k$ as asserted. This proves half of 
the proposition. The proof of the other half is dual to this one. 
\end{proof}

Armed with this proposition, we can prove the main theorem for this 
example. 

\begin{theorem} \label{thm:dihedral}
Suppose that $G$ is a finite group with a 
dihedral Sylow 2-subgroup of order at least 8, 
and that $k$ is a field of characteristic 2. 
Suppose that $\calD$ is a component of the 
Auslander-Reiten quiver of $kG$ that contains a 
module of odd dimension. Then for any $n > 0$, there exists a 
natural transformation $\psi : {\rm Id} \to \Omega^{2n-1}$ in the 
stable category $\stmod(kG)$ with the property that $\psi$
is supported only on the set of modules in $\calD$.
\end{theorem}

\begin{proof}
  Let $M$ be a module in $\calD$ having odd dimension.
  Recall that 
the collection of indecomposable modules in $\calD$ coincides with 
the collection of non-projective direct summands of modules of the 
form $M \otimes U_{2i, 2j}$ for $i$ and $j$ in $\ZZ$ \cite{AC}.
Hence, every indecomposable module in $\calD$ has odd dimension. 

Now define the natural transformation $\psi$ by the following rule.
For $M$ an indecomposable $kG$-module, let
\[
\psi_M \quad = \quad \begin{cases} \mu_{n,M} & \quad \text{if $M$ is in 
$\calD$,} \\ 0 & \quad \text{otherwise.} \end{cases}
\]
To prove the theorem we must show that, given a homomorphism 
$\varphi: M \to N$, for $M$ and $N$ indecomposable modules, the diagram
\[
\xymatrix{
M \ar[rr]^{\varphi} \ar[d]^{\psi_M} && N \ar[d]^{\psi_N} \\
\quad \Omega^{2n-1}(M) \qquad \ar[rr]^{\Omega^{2n-1}(\varphi)} && 
\qquad \Omega^{2n-1}(N) \quad
}
\]
commutes. This is clear from the definitions if either both $M$ and $N$ 
are in $\calD$ or both are not in $\calD$. 
If one of $M$ and $N$ is in $\calD$ and the other is not,
then we need only appeal to Proposition \ref{prop:vanish}. 
\end{proof}

We now show that the natural transformation just constructed  
satisfies the conditions of Theorem~\ref{finite-comp-length-theorem}, 
thereby showing that the circumstances of this theorem can 
actually arise in a non-trivial way.

\begin{proposition}
The natural transformation $\psi : \Id \to \Omega^{2n-1}$ just 
constructed satisfies the conditions of 
Theorem~\ref{finite-comp-length-theorem}. Moreover, if $f:V \to M$ 
is a map of indecomposable modules such that $\psi_M f \neq 0$,
then $f$ factors as a sum of composites of finitely many irreducible
maps. 
\end{proposition}

\begin{proof}
We know when $\calC=\stmod(kG)$ that $\tau=\Omega^2$, 
$\Sigma=\Omega^{-1}$ and $\Serre=\Omega$. Thus $\stmod(kG)$ 
is a $(-1)$-Calabi-Yau category. The fact that for each 
indecomposable $U,$ the only objects of the form $\Sigma^t U$ 
in the component of $U$ lie in the same $\tau$-orbit as $U$, 
as well as the fact that each mesh has at most 2 middle terms, 
follow from \cite{BS} and \cite{Web}. 

We show that for all indecomposable modules $M$, the natural 
transformation $\Hom_\calC(-,\psi_M)$ has finite support. When 
$M$ is not in $\calD$ this is clear, 
so we suppose $M$ lies in $\calD$. The construction of 
$\psi_M=\mu_{n,M}:M\to\Omega^{2n-1}(M)$ is that it is the 
third homomorphism in a triangle in $\stmod(kG)$ of the form
$$
\Omega^{2n}(M)\to (U_{2n,0}\oplus U_{0,2n})\otimes M\to M\to \Omega^{2n-1}(M)
$$
corresponding to a short exact sequence of $kG$-modules
$$
0\to \Omega^{2n}(M)\to (U_{2n,0}\oplus U_{0,2n})\otimes M\to M\to 0.
$$
The argument of Lemma~\ref{lemma:exact} showed that the sequence of functors
$$
\begin{aligned}
0\to \Hom_{kG}&(-,\Omega^{2n}(M))\cr
&\to \Hom_{kG}(-,(U_{2n,0}\oplus U_{0,2n})\otimes M)\to \Hom_{kG}(-,M)\cr
\end{aligned}
$$
is exact, and the final cokernel has composition factors lying in the 
diamond of the Auslander-Reiten quiver determined by the modules 
$M, \Omega^{2n}(M), U_{n,0}\otimes M$ and $U_{0,n}\otimes M$, 
including the right-hand edge of this diamond, but not the left-hand 
edge or the modules $U_{n,0}\otimes M$ and $U_{0,n}\otimes M$. This cokernel 
is the image of $\Hom_{\stmod(kG)}(-,\psi_M)$, so that condition (1) 
of Corollary~\ref{finite-composition-length} is satisfied. 
This shows that $\Hom_\calC(-,\psi_M)$ has finite support.

The final statement follows from part (5) of 
Corollary~\ref{finite-composition-length} and condition (2) of 
Proposition~\ref{almost-vanishing-characterization}. 
If $\psi_Mf\ne 0$ then $\psi_Mf\xi$ is almost vanishing for 
some $\xi$. Thus $f\xi$ is a sum of composites of irreducible 
morphisms and hence so is $f$.
\end{proof}

\end{document}